\documentclass[reqno,oneside,12pt]{amsart}
\usepackage{amssymb,enumerate,bm}

\usepackage[T1]{fontenc}
\usepackage[width=14cm,centering]{geometry}
\usepackage{tikz}
\usetikzlibrary{lindenmayersystems}
\newtheorem{theorem}{Theorem}[section]
\newtheorem{thm}[theorem]{Theorem}

\newtheorem{cor}[theorem]{Corollary}
\newtheorem{lem}[theorem]{Lemma}

\theoremstyle{definition}

\newtheorem*{ack}{Acknowledgement}
\theoremstyle{remark}
\newtheorem{remark}[theorem]{Remark}

\newtheorem{claim}[theorem]{Claim}

\newif\ifcnote\cnotefalse

\cnotetrue

\newcommand{\bbN}{\mathbb{N}}

\newcommand{\bbR}{\mathbb{R}}

\newcommand{\clC}{\mathcal{C}}

\newcommand{\clH}{\mathcal{H}}

\newcommand{\diam}{\mathrm{diam}}

\newcommand{\dist}{\mathrm{dist}}

\def\R{\mathbb{R}}

\DeclareMathAlphabet{\pazocal}{OMS}{zplm}{m}{n}

\DeclareMathAlphabet{\pazocal}{OMS}{zplm}{m}{n}
\newtheorem*{thm*}{Theorem}

\usepackage{mathtools}
\DeclarePairedDelimiter\ceil{\lceil}{\rceil}

\usepackage{hyperref}
\usepackage{parskip}
\hypersetup{
	colorlinks=true,
	linkcolor=blue,
	filecolor=magenta,      
	urlcolor=cyan,
}

\usepackage[normalem]{ulem}

\definecolor{darkgrn}{rgb}{0, 0.75, 0}

\definecolor{eaclr}{rgb}{0.2, 0.7, 0}

\title[Asymptotics of maximum distance minimizers]
    {Asymptotics of maximum distance minimizers}
    
    \author[Enrique G Alvarado]{Enrique G Alvarado}
    \author[Louisa Catalano]{Louisa Catalano}
    \author[Tom\'as Merch\'an]{Tom\'as Merch\'an}
    \author[Lisa Naples]{Lisa Naples}

    \address{Department of Mathematics, University of California at Davis, Davis, CA 95616, USA}
    \email{ealvarado@math.ucdavis.edu}
    \address{Department of Mathematics, Clayton State University, Morrow, GA 30260, USA}
    \email{LouisaCatalano@clayton.edu}
    \address{Department of Mathematics, Clayton State University, Morrow, GA    30260, USA}
    \email{TomasMerchanRodriguez@clayton.edu}
    \address{Department of Mathematics, Fairfield University, Fairfield, CT 06824, USA}
    \email{lisa.naples@fairfield.edu}


\begin{document}

\begin{abstract}
We study the limiting behavior of $r$-maximum distance minimizers and the asymptotics of their $1$-dimensional Hausdorff measures as $r$ tends to zero in several contexts, including situations involving objects of fractal nature. 
\end{abstract}

\maketitle


\section{Introduction}\label{sec: intro}
The Maximum Distance Problem (MDP) asks to find the shortest curve whose $r$-neighborhood contains a given set. 
In particular, given some compact subset $E$ of $\bbR^n$ and some radius $r > 0$, the {\it Maximum Distance Problem} is defined as follows:
\begin{equation}\label{eq: MDP}
\begin{cases}
    \text{minimize } \clH^1(\Gamma) \\
    \text{among rectifiable curves } \Gamma\subset \bbR^n \text{ such that } B(\Gamma, r) \supset E,
    \end{cases}
\end{equation}
where $B(A, r) := \{x \in \bbR^n : \mathrm{dist}(x, A) \leq r\}$ denotes the {\it closed $r$-neighborhood} of any subset $A$ of $\bbR^n$, and $\dist(x, A) := \inf_{y \in A}{|x - y|}$.
We let $\Lambda(E, r)$ be the infimum of that problem: 
\[
\Lambda(E, r) := \inf\{\clH^1(\Gamma) : \Gamma\subset \bbR^n \text{ is a rectifiable curve and } B(\Gamma, r) \supset E\}.
\]
In this paper, we will denote minimizers of~(\ref{eq: MDP}) by $\Gamma^\ast_r$, i.e., $\Gamma^\ast_r$ is a rectifiable curve such that $B(\Gamma^\ast_r, r) \supset E$ and $\clH^1(\Gamma^\ast_r) = \Lambda(E, r)$, and we will call $\Gamma^\ast_r$ an {\it r-maximum distance minimizer of $E$}. 

The MDP arises from the $L^\infty$ version of the \emph{average distance problem}, which was introduced in the seminal work~\cite{buttazzo2002optimal} of Buttazzo, Oudet, and Stepanov as an attempt to optimize mass transportation in cities. 
Given a population density modeled with a measure $\mu$ with bounded support, and given a maximum transportation network cost $l > 0$, the $p$-average distance problem minimizes the functional 
\begin{equation}\label{average distance problem}
 F_{\mu, l}(\Sigma) := \left(\int\dist(x, \Sigma )^p\, d\mu(x) \right)^{1/p}
\end{equation}
over all rectifiable curves $\Sigma$ such that $\mathcal{H}^1(\Sigma) \leq l$. 
Letting $p \to +\infty$, the $p$-average distance problem reduces to minimizing 
\begin{equation}\label{mdp dual}
F_E(\Sigma) := \max_{y \in E} \mathrm{dist}(y, \Sigma)
\end{equation}
over the same class, where $E$ is now the support of $\mu$. 
Minimizing (\ref{mdp dual}) can be seen as the ``dual problem'' of the MDP (as defined in (\ref{eq: MDP})). 
In fact, Miranda Jr., Paolini, and Stepanov showed that the minimizers of $F_E$ are equivalent to the minimizers of the MDP \cite{miranda2006one, paolini2004qualitative}.\footnote{In this paper, due to the equivalence of the minimizers we reverse Paolini and Stepanov's terminology for the ``Maximum Distance Problem'' and its ``dual''. 
In particular, our ``Maximum Distance Problem'' is what they refer to as the ``dual to the maximum distance minimizing problem''.}
These two papers, \cite{buttazzo2002optimal} and \cite{paolini2004qualitative}, were the foundation for an extensive amount of work on the average and maximum distance problems. 
The existence of rectifiable minimizers for both problems follow from standard compactness arguments of geometric measure theory, and a lot of work has been put into the study of their regularity and structure~\cite{alvarado2021maximum, buttazzo2002optimal, gordeev2022regularity, paolini2004qualitative, santambrogio2005blow, teplitskaya2021regularity}.  

The problem of studying the asymptotics of average distance minimizers (as $l \to +\infty$) for $p = 1$ was raised by Buttazzo, Oudet, and Stepanov \cite[Problem~3.6]{buttazzo2002optimal}. 
They provided a partial answer to their problem for the case that $p = 1$ and $\mu \ll \mathcal{L}^n$ ($\mu$ is absolutely continuous with respect to the Lebesgue measure $\mathcal{L}^n$) by showing that $\min_\Sigma F_{\mu, l}(\Sigma)$ is $O(l^{-1/(n-1)})$ as $l \to +\infty$. 
In 2005, Mosconi and Tilli \cite{MT2005} solved the problem for any $1 \leq p < +\infty$ and any $\mu \ll \mathcal{L}^n$. 

In this paper, we study the asymptotics for the case $p=+\infty$, and achieve interesting results without the absolute continuity assumption, which involve objects of fractal nature.

For the remainder of the introduction, we proceed to briefly outline the main results of the paper. 
We begin by showing in the theorem below that whenever $E$ is contained in a rectifiable curve, \emph{solutions to the Analyst's Traveling Salesman Problem} can be obtained as the limit (under Hausdorff convergence) of $r$-maximum distance minimizers as $r \to 0^+$. We stress that while the Analyst's Traveling Salesman Problem is not classically viewed as a minimization problem, we use the expression ``solutions to the Analyst's Traveling Salesman Problem'' to mean curves of least possible $\clH^1$ measure.
{
\renewcommand{\thetheorem}{\ref{thm: neigh MDP to ATSP}}
\begin{thm}
Suppose $E$ is contained in a rectifiable curve and let $r_j \to 0^+$. 
There exists a sequence $\{\Gamma^\ast_{r_j}\}$ of $r_j$-maximum distance minimizers of $B(E, r_j)$ and a rectifiable curve $\Gamma^\ast_0 \supset E$ such that 
\begin{enumerate}
\item[(a)] $\displaystyle\lim_{j\to \infty} d_H(\Gamma^\ast_{r_j},  \Gamma^\ast_0) = 0$ and
\item[(b)] $ \displaystyle\lim_{j\to\infty}\clH^1(\Gamma_{r_j}^\ast)  =\clH^1(\Gamma^\ast_0) = \inf\{\clH^1(\Gamma) : \Gamma \text{ is a rectifiable curve, } \Gamma \supset E\}$.
\end{enumerate}
\end{thm}
}
The expression $d_H(\cdot,\cdot)$ refers to the Hausdorff distance, see Section \ref{sec: MDP_ATP} for the definition.

We are also interested in studying how minimizers from the Maximum Distance Problem, which are defined as 1-dimensional sets, behave when the set to approximate has some rough geometry. 
The natural context to study such interaction is in the setting of H\"older curves (see Section \ref{sec: MDP Holder} for a precise definition of H\"older curves). 
In the theorem below, we achieve precise estimates which quantify these relationships. 

{
\renewcommand{\thetheorem}{\ref{thm: general case}}
\begin{thm}
Let $0 < \alpha \leq \beta$ and let $\gamma:[0,1]\to \bbR^n$ be a \textit{weak} $(\alpha,\beta)$-bi-H\"older curve with constant $1\leq C_\gamma < \infty$. 
Then there exists $C=C(\alpha, \beta, C_\gamma,n)$ such that 
\[
\frac{1}{C} r^{\frac{\beta-1}{\beta}} \leq \Lambda(B(\widehat\gamma,r), r) \leq C r^{\frac{\alpha-1}{\alpha}},
\]
for all small enough $r = r(\beta,C_\gamma) > 0$. 
\end{thm}
}

It is worth mentioning that the study of H\"older curves and its analogous concept of rectifiability have attracted plenty of attention lately as a higher dimensional alternative to rectifiable curves. See, for example, \cite{BNV19,BV2019}.

In Theorem \ref{thm: neigh MDP to ATSP}, we already show (using compactness arguments) that in the case in which $E$ can be covered by a rectifiable curve $\Gamma^\ast_0$, there exists a sequence of minimizers that converge in Hausdorff distance to $\Gamma^\ast_0$. In more general situations, compactness arguments are not available and other techniques are required.
In the results below, we succeed in overcoming such obstacles.
{
\renewcommand{\thetheorem}{\ref{thm: Hausdorff conv H1}}
\begin{thm}
     Let $E \subset \R^2$ be path connected. 
     For any $r>0$, we let $\Gamma^\ast_r$ be an $r$-maximum distance minimizer of $B(E, r)$. 
     Then  $$B(\Gamma^\ast_r,r)\subset B(E,Cr),$$
     where $C>1$ is an absolute constant.
\end{thm}
}

{
\renewcommand{\thetheorem}{\ref{cor: Hausdorff conv length}}
\begin{cor}
    Under the assumptions above, we have that $\Gamma^\ast_r$ converges to $E$ in Hausdorff distance as $r \rightarrow 0^+$.
\end{cor}
}

The proof of Theorem \ref{thm: Hausdorff conv H1} is obtained  (via contradiction) by methodically analyzing the geometry of any minimizer $\Gamma^\ast_r$ not contained in a large ball around $E$. 

\begin{remark}
    As we proceed through the results, we find that many can be restated with the set $E$ in place of $B(E,r)$ (or $\widehat\gamma$ in place of $B(\widehat\gamma,r)$). The corresponding proofs can be achieved without significant alterations. More details can be found in Remarks \ref{rem:neigh MDP to ATSP B(E,r) to E}, \ref{rem:general case B(E,r) to E}, and \ref{rem:Convergence of minimizers B(E,r) to E}.
\end{remark}

\begin{ack}
The authors are grateful to the Institute of Advanced Study (IAS) for their hospitality. The IAS: Summer Collaborator Program provided invaluable support  for this project, which was partially completed while staying at IAS.
\end{ack}

 \section{Maximum distance minimizers and the Analyst's Traveling Salesman Problem}\label{sec: MDP_ATP}

For any subset $E\subset \bbR^n$, we will let $F_E$ be the non-negative function defined over subsets of $\bbR^n$ by the formula $F_E(A) := \max_{y \in E}\dist(y, A)$.
Additionally, we remind the reader about the definition of Hausdorff distance: given two sets $A,B \subset \R^n$, we set
$$d_H(A,B)=\max\{F_A(B),F_B(A)\}.$$
The following theorem shows that if $E$ is contained in a rectifiable curve, then the limit as $r$ tends to $0^+$ of $r$-maximum distance minimizers of $B(E, r)$ is a rectifiable curve containing $E$.  
In particular, the sequence of minimizers specifies a  solution curve to the Analyst's Traveling Salesman Problem for the set $E$.

\begin{thm}\label{thm: neigh MDP to ATSP}
Suppose $E$ can be covered by a rectifiable curve and let $r_j \to 0^+$. 
There exists a sequence $\{\Gamma^\ast_{r_j}\}$ of $r_{j}$-maximum distance minimizers of $B(E, r_j)$ and a rectifiable curve $\Gamma^\ast_0 \supset E$ such that 
\begin{enumerate}
\item[(a)] $\displaystyle\lim_{j\to \infty} d_H(\Gamma^\ast_{r_j},  \Gamma^\ast_0) = 0$ and 
\item[(b)] $ \displaystyle\lim_{j\to\infty}\clH^1(\Gamma_{r_j}^\ast)  =\clH^1(\Gamma^\ast_0) = \inf\{\clH^1(\Gamma) : \Gamma \text{ is a rectifiable curve, } \Gamma \supset E\}$.
\end{enumerate}
\end{thm}
\begin{proof}
First notice that the limit
\begin{equation}\label{eq: boundonL_2}
L := \lim_{r \to 0+} \Lambda(B(E, r), r) \leq \inf\{\clH^1(\Gamma) : \Gamma \text{ is a rectifiable curve, } \Gamma \supset E\}
\end{equation}
exists since the map $r\mapsto \Lambda(B(E, r), r)$ is bounded and monotonic. 
In particular, $\Lambda(B(E, r), r) \leq \inf\{\clH^1(\Gamma) : \Gamma \text{ is a rectifiable curve, } \Gamma \supset E\}$ for all $r > 0$ since $B(A, r) \supset B(E, r)$ for any set $A \supset E$. 

Next, let $0 < r_i < 10$ be a sequence of real numbers converging to $0$, and let $\{\Gamma^\ast_{r_i}\}$ be a sequence of $r_i$-maximum distance minimizers of $B(E, r_i)$. 
That is, $B(\Gamma^\ast_{r_i}, r_i) \supset B(E, r_i)$ and $\clH^1(\Gamma^\ast_{r_i}) = \Lambda(B(E, r_i), r_i)$ for all $i \in \bbN$.
Since $E$ is bounded, there is some large enough ball $B(0, R - 20)$ containing $B(E, 10)$. 
We may assume that $\Gamma^\ast_{r_i} \subset B := B(0, R)$ for all $i \in \bbN$ since if they were not, then $\clH^1(\pi_B(\Gamma^\ast_{r_i})) \leq \clH^1(\Gamma^\ast_{r_i})$ and $B(\pi_B(\Gamma^\ast_{r_i}), r_i) \supset B(E,r_i)$. 
Additionally, we may assume that $\sup_i \clH^1(\Gamma^\ast_{r_i}) < \infty$. 

By the Blaschke selection theorem there exists a subsequence $\{\Gamma^\ast_{r_{j}}\}$ of $\{\Gamma^\ast_{r_i}\}$ and a compact set $\Gamma^\ast_0$ such that $\Gamma^\ast_{r_{j}} \xrightarrow{d_H} \Gamma^\ast_0$ as $j \to \infty$.
In addition, since $F_C(A) \leq F_B(A) + F_C(B)$ for any $A,B, C\subset \bbR^n$, and $B(E, r_{j}) \xrightarrow{d_H} E$ as $j \to \infty$, we get that
\begin{align}
F_E(\Gamma^\ast_0) &\leq F_{\Gamma^\ast_{r_{j}}}(\Gamma^\ast_0) + F_{B(E, r_{j})}(\Gamma^\ast_{r_{j}}) + F_E(B(E, r_{j}))\nonumber\\
&\leq F_{\Gamma^\ast_{r_{j}}}(\Gamma^\ast_0) + r_{j} + F_E(B(E, r_{j})) \to 0\label{eq: hausdorff convergence}
\end{align}
as $j \to \infty$. 
Thus, $F_E(\Gamma^\ast_0) = 0$ and hence $\Gamma^\ast_0 \supset E$. 
Now let us show that $\clH^1(\Gamma^\ast_0) = L$.
First we prove that $\clH^1(\Gamma^\ast_0) \leq L$. 
Since each $\Gamma^\ast_{r_j}$ is connected, Golab's theorem implies that $\Gamma^\ast_0$ is connected and 
\[
\clH^1(\Gamma^\ast_0) \leq \liminf_{j\to\infty}\clH^1(\Gamma^\ast_{r_{j}}) = \lim_{j\to \infty}\clH^1(\Gamma^\ast_{r_{j}}) = L.
\]
To show that $\clH^1(\Gamma^\ast_0) \geq L$, simply notice that $\clH^1(\Gamma^\ast_0) \geq \clH^1(\Gamma^\ast_{r_i})$ for all $i$ since $\Gamma^\ast_0$ is a rectifiable curve and since $B(\Gamma^\ast_0, r) \supset B(E, r)$ for all $r > 0$.
Therefore 
\begin{equation}\label{eq: LequalsH1_2}
L = \clH^1(\Gamma^\ast_0).
\end{equation}
Thus by ~(\ref{eq: boundonL_2}) and ~(\ref{eq: LequalsH1_2}), $\clH^1(\Gamma^\ast_0) = \inf\{\clH^1(\Gamma) : \Gamma \text{ a rectifiable curve, } \Gamma \supset E\}$.
\end{proof}

\begin{remark}\label{rem:neigh MDP to ATSP B(E,r) to E}
We may replace $B(E, r)$ with $E$ in the statement of Theorem~\ref{thm: neigh MDP to ATSP}, and still obtain the same conclusion.
The proof follows the proof of Theorem~\ref{thm: neigh MDP to ATSP} almost verbatim.
We simply need to replace inequality~(\ref{eq: hausdorff convergence}) with:
\[
F_E(\Gamma^\ast_0) \leq F_{\Gamma^\ast_{r_{j}}}(\Gamma^\ast_0) + F_E(\Gamma^\ast_{r_{j}}) \leq F_{\Gamma_{r_{j}}^\ast}(\Gamma^\ast_0) + r_{j} \to 0.
\]
\end{remark}

\section{The Maximum Distance Problem for H\"older curves}\label{sec: MDP Holder}

In this section, we provide measure bounds for $r$-maximum distance minimizers of $(\alpha, \beta)$-bi-H\"older curves as stated in Theorem \ref{thm: general case}.

For $\alpha \in (0, 1]$, we say that a map $\gamma : [0, 1] \to \bbR^n$ is an \emph{$\alpha$-H\"older curve} with constant $1 \leq  C_\gamma < \infty$  if 
\[ 
|\gamma(x)-\gamma(y)|\leq C_\gamma |x-y|^{\alpha} \text{  for all }x,y\in [0,1].
\]
 We denote the class of $\alpha$-H\"older curves by $\mathcal{C}^{0,\alpha}([0,1])$.  
If in addition, there exists a $\beta \geq \alpha$ such that 
\[
\frac{1}{C_\gamma} |x - y|^\beta \leq |\gamma(x)-\gamma(y)|\leq C_\gamma |x-y|^{\alpha} \text{  for all } x,y\in [0,1],
\]
 then we call $\gamma$ an \emph{$(\alpha, \beta)$-bi-H\"older curve with constant $C_\gamma$}, or simply an \emph{$(\alpha, \beta)$-bi-H\"older curve}. 
If $\beta = \alpha$ above, we simply say that $\gamma$ is an \emph{$\alpha$-bi-H\"older curve}.
Additionally, for any curve $\varphi:[0,1]\to\mathbb{R}^n$, we will denote $\varphi([0,1])$ by $\widehat{\varphi}$.

There are several known  $\alpha$-bi-H\"older curves, among them, the von Koch snowflake (with $\alpha=\log_4(3))$.
In the literature, $\alpha$-bi-H\"older curves are often presented as Lipschitz maps from the unit interval equipped with the snowflake metric into $\bbR^n$; see e.g. \cite{Koskla_94} and the references listed therein.

The study of the Maximum Distance Problem is more complicated whenever the set $E$ cannot be covered by a rectifiable curve, e.g. $E$ is the von Koch snowflake. For instance, we cannot use compactness arguments since the function $r \mapsto \Lambda(B(E, r), r)$ is unbounded. 
In order to understand the asymptotic behaviour of minimizers for more general sets, it is natural to start our study in the context of H\"older curves.

In order to introduce the general techniques in a more familiar context, we begin such a study in Section \ref{sec: neighborhoodsOfVonKoch} by first looking at $r$-maximum distance minimizers of $r$-neighborhoods of the von Koch snowflake, which will be denoted by $S$ (depicted in Figure \ref{fig: snowflake}). 
In Lemma~\ref{lem: snowflake}, and its Corollary~\ref{cor: snowflake corollary}, we establish precise estimates for  the behaviour of the function $r \mapsto \Lambda(B(S, r), r)$. 

{
\renewcommand{\thetheorem}{\ref{cor: snowflake corollary}}
\begin{cor}
Let $S$ be the $\frac{1}{3}$-von Koch snowflake.
There exists a constant $C > 1$ such that
\[
\frac{1}{C}r^{\frac{\alpha-1}{\alpha}} \leq \Lambda(B(S, r), r) \leq C r^{\frac{\alpha-1}{\alpha}},
\]
for all $0<r < 1$ where $\alpha=\log_4(3)$.
\end{cor}
}

In fact, when the authors started this line of investigation, Corollary \ref{cor: snowflake corollary} was one of the first results obtained. The result itself, and more importantly, the intuition behind proof, informs the more general Theorem \ref{thm: general case}.

\subsection{Elementary measure bounds}

\begin{lem}\label{lem: minimizerof2balls}
Let $a,b\in\bbR$. If $E = B(a, r)\cup B(b, r)$, then the line segment $[a, b]$ is an $r$-maximum distance minimizer of $E$ for any $r > 0$. 
\end{lem}
\begin{proof}
Let $r > 0$.
First, notice that $B([a, b], r) \supset E$ since the line segment contains the points $a$ and $b$. 
Now, without loss of generality, assume that the line segment $[a, b]$ lies on the first coordinate axis such that $a = (0,\dots, 0)$, and $b = (|a - b|, 0, \dots, 0)$. 
Any $r$-maximum distance minimizer, $\Gamma^\ast_r$ must have a non-empty intersection with the two half-spaces $H_1 := \{(x_1, \dots, x_n) : x_1 \leq 0\}$ and $H_2 := \{(x_1, \dots, x_n) : x_1 \geq |a - b|\}$, otherwise, $B(\Gamma^\ast_r, r)$ would not contain $E$.
Since 
\[
\inf\{|x - y| : x \in H_1, y \in H_2\} = |a - b| = \clH^1([a, b]),
\]
the line segment $[a, b]$ is an $r$-maximum distance minimizer of $E$. 
\end{proof}

\begin{lem}\label{lem:diam estimate}
Given a compact set $E$, we have that    for all  $r >0$,
\[
 \Lambda(B(E, r), r) \geq \diam(E) .
\] 
\end{lem}
\begin{proof}
Since $E$ is a compact set, there exist two points $a, b \in E$ such that $|a - b| = \diam(E)$. 
Let $B_a := B(a, r)$ and let $B_b := B(b, r)$. 
By Lemma~\ref{lem: minimizerof2balls}, $[a, b]$ is an $r$-maximum distance minimizer of $B_a \cup B_b$, and hence 
\[
 \diam(E) = |a - b| = \Lambda(B_a\cup B_b, r) \leq \Lambda(B(E,r), r),
\]
where the last inequality is due to the fact that $\Lambda(E_1, r) \leq \Lambda(E_2, r)$ whenever $E_1 \subset E_2$. 
\end{proof}


\subsection{Investigating the von Koch Snowflake as a case study}\label{sec: neighborhoodsOfVonKoch}

First, we briefly estimate $\Lambda(B(S,r),r)$ at large scales before turning to the small scale case.

\pgfdeclarelindenmayersystem{Koch curve} 
{\rule {F -> F+F--F+F}}

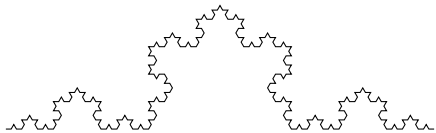
\begin{figure}
    \centering
    \begin{tikzpicture}
        \draw[l-system={Koch curve, step=2pt, angle=60, axiom=F, order=4}] lindenmayer system;
    \end{tikzpicture}
    \caption{The standard von Koch snowflake}
    \label{fig: snowflake}
\end{figure}

\begin{lem}\label{SnoflakeUnitScale}
    Let $\frac13 \leq r< 1$. Then $1\leq \Lambda(B(S,r), r)\leq  3$. 
\end{lem}
\begin{proof}
    The lower bound follows from the fact that $\diam(S)\geq 1$ together with Lemma \ref{lem:diam estimate}. Let us prove the upper bound. 
Firstly, elementary observations show  that 
$$\pi_{1}(S)\subset [0,1] \quad \text{ and } \quad 0\leq\pi_{2} \, (x)\leq\frac{1}{3},$$
for every $x \in S$, where $\pi_i$ is the standard projection onto the $i$th coordinate ($i=1,2$).
Consequently, we have that for any $r \geq \frac{1}{3}$, \[B(S,r) \subset B\Big(\partial \big([0,1]\times \big[0,\tfrac{1}{3}\big]\big),r\Big).\]  Since 
$\partial([0,1]\times [0,\frac{1}{3}])$ is a Lipschitz curve with $\clH^1$-measure less then $3$, we obtain the lemma.
\end{proof}

For the investigation at scales $0< r< \frac13$, we use the inherent self-similarity of the von Koch snowflake. For that purpose, it is useful to have an appropriate scale of detail -- depending on the neighborhood size -- at which to view a H\"older curve. The following claim will help us to do so.
\begin{claim}\label{claim:k_r}
    For $0<r<\frac13$, there exists $k_r\in\mathbb{Z}^+$ and $\frac{1}{3}\leq r'<1$ such that $r'=3^{k_r}r$. 
\end{claim}
    Indeed, take 
        $$k_0=\sup \left\{k\in\mathbb{Z} : 3^kr<\frac13\right\}.$$
    We note that if $3^{k_0}r< \frac19$, then $3^{k_0+1}r< \frac13$, which is a contradiction to $k_0$ being the supremum. Therefore, we have $\frac19\leq 3^{k_0}r<\frac13$. This implies that $\frac13 \leq 3^{k_0+1}r<1$, and so we take $k_r=k_0+1$.

\begin{lem}\label{lem: snowflake}
    There exists $C>1$ such that for any $0<r<\frac13$, we have 
    $$\frac1C \Lambda(B(S,r'), r') \Big(\frac{4}{3}\Big)^{k_r}\leq \Lambda(B(S,r), r) \leq C\Lambda(B(S,r'), r')\Big(\frac{4}{3}\Big)^{k_r},$$
    where $r'$ and $k_r$ are as in Claim \ref{claim:k_r}. 
\end{lem}

\begin{proof}
We will first prove the lower bound.  In the case  $k_r=1$ or $k_r=2$, we simply observe that $\diam(S)\ge 1$.  Then, via Lemma \ref{lem:diam estimate}, we immediately obtain \[\Lambda(B(S,r),r)\ge \diam(S)\ge 1\ge \frac{1}{3}\Lambda(B(S,r'),r');\] the last inequality follows from Lemma \ref{SnoflakeUnitScale}. 

We now deal with the case $k_r\ge 3$. Assume  that there exists a rectifiable curve $\Gamma$ with $B(\Gamma, r)\supset B(S,r)$ and $\mathcal{H}^1(\Gamma) < \frac1L \Lambda(B(S,r'), r') \big(\frac43 \big)^{k_r}$, for some $L>0$ big enough to be determined later. 
By self-similarity,  we have $4^{k_r-1}$ distinct copies of the snowflake $S$ of size $3^{-k_r+1}$ within $S$.  We divide $S$ into groups of 4 consecutive  size $3^{-k_r+1}$ snowflakes;  there are $4^{k_r-2}$ such groups. Going from left to right, we label the groups as $\widetilde{A}_{1}, \widetilde{A}_{2}, ..., \widetilde{A}_{4^{k_r-2}}$. For a particular group $\widetilde{A}_{j}$, the four snowflakes in the group will be labeled $A_{j,1}$, $A_{j,2}$, $A_{j,3}$, $A_{j,4}$. Additionally, we define 
$$\Gamma_{\widetilde{A}_{j}}=\{x\in \Gamma : B(x,r)\cap \widetilde{A}_{j} \neq\varnothing\},$$
and 
$$\Gamma_{A_{j,i}}=\{x\in \Gamma : B(x,r)\cap A_{j,i} \neq\varnothing\}, $$
for $j=1,...,4^{k^r-2}$  and $ i=1,...,4$.

\begin{claim}
\label{clm:no_overlap}
    $\Gamma_{A_{j,2}} \cap \Gamma_{\widetilde{A}_{m}} = \varnothing$ when $j\neq m$.
\end{claim}

To verify this claim, it is enough to assume that $\widetilde{A}_m$ is the sub-snowflake of $S$ immediately to the right of $\widetilde{A}_{j}$. 
 We consider covers of sub-snowflakes of $S$ by triangles that are dilations and rotations of the triangle with vertices $(0,0)$, $(1/2, \sqrt{3}/{6})$, and $(1,0)$.  In particular, to cover $\widetilde{A}_m$ we use a triangle $\triangle_m$ at scale $3^{-k_r+2}$ and to cover each of $A_{j_1}$ and $A_{j,2}$ we use triangles $\triangle_{j,1} $ and $\triangle_{j,2}$ at scale $3^{-k_r+1}$ as in Figure \ref{fig: separation}. The distance between $\widetilde{A}_m$ and $A_{j,2}$ is at least $\diam(\triangle_{j,1})=3^{-k_r+1}$. Now the desired result holds as long as $2r<3^{-k_r+1}$ or, equivalently, $3^{k_r}r<\frac{3}{2}$.   But by definition, $3^{k_r}r<1$.  Therefore the claim is proven.

 \begin{figure}
    \centering
     \includegraphics[scale=.75]{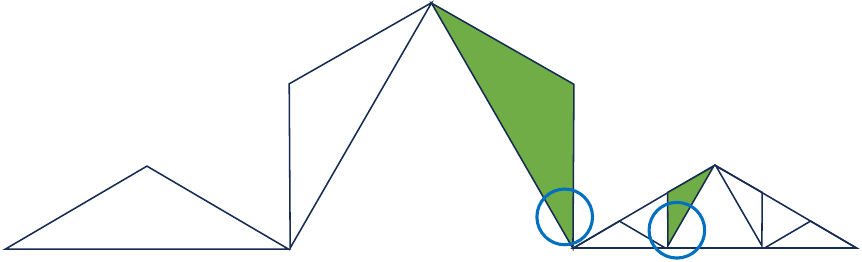}
     \caption{$A_{j,2}$ and $\widetilde{A}_m$ are sufficiently separated so that a ball of radius $r$ centered in $A_{j,2}$ cannot intersect a ball of radius $r$ centered in $\widetilde{A}_m$.}
    \label{fig: separation}
 \end{figure}

Now, for each $j=1,2,\dots, 4^{k_r-2}$, we define a subset of $\widetilde A_j$ that will not overlap with distinct $\widetilde A_m$:
$$\Gamma_{j,\text{no}} := \{ x\in\Gamma : B(x,r)\cap \widetilde{A}_j \neq \varnothing \text{ and } B(x,r)\cap \widetilde{A}_{m} = \varnothing \text{ for } j\neq m \}.$$

We observe the union, 
$$\Gamma_{\text{no}} := \bigcup_{j=1}^{4^{k_r-2}} \Gamma_{j,\text{no}},$$
is disjoint.
Hence, we have that
$$ \mathcal{H}^1(\Gamma) \geq  \mathcal{H}^1(\Gamma_{\text{no}}) \geq\sum_{j=1}^{4^{k_r-2}} \mathcal{H}^1(\Gamma_{j,\text{no}}) \geq \sum_{j=1}^{4^{k_r-2}} \mathcal{H}^1(\Gamma_{A_{j,2}}),$$
where the last inequality follows by Claim \ref{clm:no_overlap}.
From this, we can see that there exists $j$ such that $ \mathcal{H}^1(\Gamma_{A_{j,2}}) < \frac{16}{4^{k_r}}\mathcal{H}^1(\Gamma)$ and since 
$$\frac{1}{L}\Lambda(B(S,r'), r')\frac{4^{k_r}}{3^{k_r}} > \mathcal{H}^1(\Gamma),$$ 
we have $\mathcal{H}^1(\Gamma_{A_{j,2}})<\frac{16}{L}\Lambda(B(S,r'), r')\frac{1}{3^{k_r}}$. By picking $L>0$ big enough, and using the self-similarity of $S$ to scale upwards by a factor of $3^{k_r-1}$, we obtain a contradiction with the lower bound of $\Lambda(B(S,r'), r')$ provided by Lemma \ref{SnoflakeUnitScale}.

We proceed now to obtain the upper bound using an extension of the technique from Lemma \ref{SnoflakeUnitScale}. We place rotated copies of $\partial ([0,3^{-k_r}]\times [0,3^{-(k_r+1)}])$ on the $4^{k_r}$ copies of the snowflake of size $3^{-k_r}$ that form $S$ as in Figure \ref{fig: rectangles}. 
Finally, we notice that the union of the boundary of such rectangles gives us a connected set $\Gamma$ whose $\clH^1$ measure is less than $\displaystyle 3 (4/3)^{k_r}$ and satisfies that $B(S,r) \subset B(\Gamma,r)$.
\end{proof}

\begin{figure}[h]
    \centering
    \includegraphics[scale=.75]{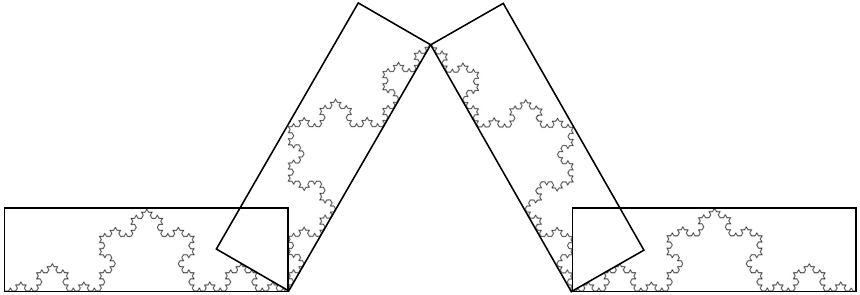}
    \caption{$\partial ([0,3^{-k_r}]\times [0,3^{-(k_r+1)}])$ curves around copies of scale $3^{-k_r}$ snowflakes.}
    \label{fig: rectangles}
\end{figure}

\begin{cor}\label{cor: snowflake corollary}
Let $S$ be the $\frac{1}{3}$-von Koch snowflake.
There exists a constant $C > 1$ such that
\[
\frac{1}{C}r^{\frac{\alpha-1}{\alpha}} \leq \Lambda(B(S, r), r) \leq C r^{\frac{\alpha-1}{\alpha}},
\]
for all $0<r < 1$ where $\alpha=\log_4(3)$.
\end{cor}

\begin{proof}
It is easy to see that we only have to consider $0<r<\frac{1}{3}$. First, we apply Lemma \ref{SnoflakeUnitScale} to rewrite the result of Lemma \ref{lem: snowflake} as 
$$\frac{1}{C}\left(\frac{4}{3}\right)^{k_r} \leq \Lambda(B(S, r), r) \leq C \left(\frac{4}{3}\right)^{k_r}.$$ 
Next, we write $k(r)=-\log_3(r)-1$ and notice that \[-\log_3(r)-1\le\lceil k(r)\rceil \le -\log_3(r),\]
so that 
\[\frac{1}{3r}\le 3^{\lceil k(r)\rceil}\le \frac{1}{r}.\]
This implies that $\lceil k(r)\rceil$ is the integer $k_r$ specified in Claim \ref{claim:k_r}.
Lastly, let us show that we may remove the ceiling function from the definition of $k(r)$.
If $a=\frac{4}{3}$, letting $C' = Ca$ gives us:
\[
\frac{1}{C'}a^{k(r)} = \frac{1}{C}a^{k(r) - 1} \leq \frac{1}{C}a^{\ceil{k(r)}} \leq  \Lambda(B(S, r), r) \leq Ca^{\ceil{k(r)}} \leq  Ca^{k(r) + 1} = C'a^{k(r)}. 
\]
Equivalently,
\[
\frac{1}{C''} a^{-\log_3(r)} \leq  \Lambda(B(S, r), r) \leq C''a^{-\log_3(r)} 
\]
where $C'' := C'a$.
Lastly, elementary algebra shows that 
$$\Big(\frac{4}{3}\Big)^{-\log_3(r)}=r^{\frac{\alpha-1}{\alpha}}, \text{ with }\alpha=\log_4(3),$$
concluding the proof of the corollary.\end{proof}

\subsection{Measure bounds in the General bi-H\"older curves}
We now generalize the results for the  von Koch snowflake to arbitrary $(\alpha, \beta)$-bi-H\"older curves. For the reader's convenience, we restate the result that we set out to prove.  The theorem then will swiftly follow from two subsequent lemmas, the first of which (Lemma \ref{lem: lower bound}) establishes the lower bound and the second of which (Lemma \ref{lem: upper bound}) establishes the upper bound.

\begin{thm}\label{thm: general case}
Let $0 < \alpha \leq \beta$ and let $\gamma:[0,1]\to \bbR^n$ be a $(\alpha,\beta)$-bi-H\"older curve with constant $1\leq C_\gamma < \infty$. 
Then there exists $C=C(\alpha, \beta, C_\gamma,n)$ such that 
\[
\frac{1}{C} r^{\frac{\beta-1}{\beta}} \leq \Lambda(B(\widehat\gamma,r), r) \leq C r^{\frac{\alpha-1}{\alpha}},
\]
for all small enough $r = r(\beta,C_\gamma) > 0$.
\end{thm}

As a consequence of this, we obtain the following important corollary, estimating the asymptotic behavior for the values 
\begin{equation}\label{parameterized MDP}
\widehat{\Lambda}(E, r) := \inf\{\ell(\gamma) : \gamma \in \clC^{0, 1}([0, 1]) \text{ such that } B(\widehat{\gamma}, r) \supset E\}.
\end{equation}
The problem above is a variation of the MDP, where we minimize the \emph{length} of parameterizations rather than the $\mathcal{H}^1$ measure of their images.
The \emph{length} $\ell(\gamma)$ of a curve $\gamma$ is defined by 
$$\ell(\gamma) := \sup \sum_{k = 1}^N |\gamma(t_k) - \gamma(t_{k-1})|,$$ where the supremum is taken over all partitions $a = t_0 < t_1 < \cdots < t_{N-1} < t_N = b$ of $[a, b]$. 

\begin{cor}{\label{cor: general case}}
Let $0 < \alpha \leq \beta$ and let $\gamma:[0,1]\to \bbR^n$ be a $(\alpha,\beta)$-bi-H\"older curve with constant $1\leq C_\gamma < \infty$. 
Then there exists $C=C(\alpha, \beta, C_\gamma,n)$ such that 
\[
\frac{1}{C} r^{\frac{\beta-1}{\beta}} \leq \widehat{\Lambda}(B(\widehat\gamma,r), r) \leq C r^{\frac{\alpha-1}{\alpha}},
\]
for all small enough $r = r(\beta,C_\gamma) > 0$.
\end{cor}

\begin{remark}
The existence of solutions to (\ref{parameterized MDP}), i.e., the existence of a Lipschitz curve $\gamma^\ast : [0, 1] \to \bbR^n$ such that $B(\widehat{\gamma^\ast}, r) \supset E$ and $\ell(\gamma^\ast) = \widehat{\Lambda}(E, r)$ can be proven with the Arzela-Ascoli theorem together with the fact that $\ell$ is lower-semicontinuous with respect to pointwise convergence. 
\end{remark}

We proceed to explore the proof of Theorem \ref{thm: general case}. Analogously to when we investigated the snowflake, it turns out to be  useful to have an appropriate scale of detail at which to view a H\"older curve. For that purpose, we develop the following definition.

Let $\eta >0$. For $0<r<1$, we define the \textit{$r$-appropriate scale} $k_\eta(r)$ as
$$k_\eta(r) := - \log_{2^\eta}(r).$$

\begin{lem}\label{lem: lower bound}
Let $0 < \alpha \leq \beta$ and let $\gamma:[0,1]\to \bbR^n$ be a $(\alpha,\beta)$-bi-H\"older curve with constant $1\leq C_\gamma < \infty$. Then there exists  $C=C(\beta,C_\gamma)>0$ such that 
    \begin{equation} \label{lowerbound}
    \Lambda(B(\widehat\gamma,r),r) \geq C r^{\frac{\beta-1}{\beta}},
    \end{equation}
    for all small enough $r=r(\beta, C_\gamma)>0$. 
\end{lem}

\begin{proof}

Assume that the conclusion (\ref{lowerbound}) does not hold.  That is, assume that there exists a $(\alpha,\beta)$-bi-H\"older curve $\gamma:[0,1] \rightarrow \mathbb{R}^n$ with constant $C_\gamma \geq 1$ and a curve $\varphi: [0,1] \to \bbR^n$ satisfying $B(\widehat\gamma,r) \subset B(\widehat\varphi,r)$ and
$$ \mathcal{H}^1(\widehat\varphi)<\frac{1}{L}C_{\gamma}^{-1/\beta} r^{\frac{\beta-1}{\beta}}=\frac{1}{L} C_{\gamma}^{-1/\beta} 2^{k_\beta(r)(1-\beta)},$$
for some $L>0$ large enough to be determined later.
We start by splitting the domain of $\gamma$ by setting
    $$  t_i= \left( 2^{k_\beta(r)}  C_{\gamma}^{-1/\beta}\right)^{-1}i, \text{ for } i=0,...,\lfloor C_{\gamma}^{-1/\beta}2^{k_\beta(r)}\rfloor.$$ 
Next, we take the union of $m$ consecutive intervals, where $m\geq 1$ will be determined later, and we denote the images of the unions as $$A_l=\gamma([t_{(l-1)m},t_{lm}]), \text{ for }l=1,...,\lfloor \frac{ 2^{k_\beta(r)}  C_{\gamma}^{-1/\beta}}{m} \rfloor.$$
   
Additionally, we define cores of the images as 
$$\widetilde{A}_l=\gamma\Big(\Big[\frac{2t_{(l-1)m}+t_{lm}}{3}, \frac{t_{(l-1)m}+2t_{lm}}{3}\Big]\Big).$$

For any $t \in [0,1] \setminus [t_{(l-1)m},t_{lm}]$, we have that 
\begin{equation}\label{distance}
\dist(\gamma(t), \widetilde{A}_l)\geq \frac{1}{3^{\beta}} 2^{-k_\beta(r)\beta} m^\beta >2r,
\end{equation}
by assuming $m$ to be large enough. Notice that, due to the definition of the $r$-appropriate scale, $m$ is chosen only in terms of $\beta$. Moreover, at this time we assume that $r$ is small enough so that 
\begin{equation}\label{r-restriction}
    C_\gamma^{-1/\beta} 2^{k_\beta(r)} > 2m.
\end{equation}
As consequence of (\ref{distance}), we obtain that the sets $\widetilde{A}_l$ are disjoint with the uniform lower bound of $2r$ on the separation distance. 

Next, we define sets $A'_l=\widehat\varphi\cap B(\widetilde{A}_l,r).$
Due to equation (\ref{distance}),  we have once more that the sets $A'_l$ are disjoint, and so we obtain that 
$$ \mathcal{H}^1(\widehat\varphi)\geq \sum_{l=1}^{\lfloor\frac{2^{k_\beta(r)}C_{\gamma}^{-1/\beta}}{ m}\rfloor} \mathcal{H}^1(A'_l).$$
Using (\ref{r-restriction}), we see that there exists $l_0$ such that 
    $$ \mathcal{H}^1(A'_{l_0})<  2\frac{m}{2^{k_\beta(r)} C_{\gamma}^{-1/\beta}}\frac{1}{L} C_{\gamma}^{-1/\beta} 2^{k_\beta(r)(1-\beta)}  =\frac{2m2^{-k_\beta(r) \beta}}{L}.$$
    
    Let $\psi: \R \to \R$ be the affine function mapping from  $[0,1]$ to $$\displaystyle I:=\Big[\frac{2t_{(l_0-1)m}+t_{l_0m}}{3}, \frac{t_{(l_0-1)m}+2t_{l_0m}}{3}\Big].$$
Next, we define the curve $\gamma_0: [0,1] \to \mathbb{R}^n $ as $\gamma_0=|I|^{-\beta}(\gamma\circ \psi  )$. We notice that $\gamma_0$ is also a weak $(\alpha,\beta)$-H\"older curve with constant $C_\gamma \geq 1$. Analogously, we define a curve $\varphi_0$ such that $\widehat{\varphi_0}=|I|^{-\beta} A'_{l_0}$. Of course,
we have that $ B(\widehat\gamma_0,|I|^{-\beta}r) \subset B(\widehat\varphi_0,|I|^{-\beta}r) $ and that 
$$ \mathcal{H}^1(\widehat{\varphi_0})<  \frac{2\cdot 3^\beta m^{1-\beta}C_\gamma^{-1}}{L} . $$
However, since $\diam(\widehat{\gamma_0})\geq C_\gamma^{-1}$, Lemma \ref{lem:diam estimate} yields that  $$\Lambda(B(\widehat{\gamma_0},|I|^{-\beta}r),|I|^{-\beta}r)\geq C^{-1}_\gamma.$$ Therefore, picking $L$ big enough only terms in of $m$ (which was picked in terms of $\beta$), we obtain a contradiction. 
\end{proof}

Next, we obtain the corresponding upper bound. 
\begin{lem}\label{lem: upper bound}
Let $0 < \alpha \leq \beta$ and let $\gamma:[0,1]\to \bbR^n$ be a \textit{weak} $(\alpha,\beta)$-bi-H\"older curve with constant $1\leq C_\gamma < \infty$. Then there exists  $C=C(\alpha, C_\gamma,n)>0$ such that
    $$\Lambda(B(\widehat\gamma,r),r) \leq C r^{\frac{\alpha-1}{\alpha}}, $$
    for all $0<r<1$.
\end{lem}

\begin{proof}
    We start by partitioning the interval $[0,1]$ as 
    $$ t_i=  i \frac{1}{\lceil 10^{1/\alpha}C_\gamma^{1/\alpha}2^{k_\alpha(r)}\rceil }, \text{ for } i=0,...,\lceil 10^{1/\alpha}C_\gamma^{1/\alpha}2^{k_\alpha(r)}\rceil.$$
   
    For simplicity of notation, we provide details of the proof for the planar case, $\mathbb{R}^2$.
    At each point $\gamma(t_i)$, we center a circle $C_i$ of radius $\frac{r}{2}$.  The choice of $t_i$ ensures that $C_i \cap C_{i+1} \neq \varnothing$ and that $B(\widehat\gamma,r) \subset \bigcup_i B(C_i,r)$. Both are proved in similar fashion, so we provide the details for the latter.  
    
    Let $z\in B(\widehat{\gamma}, r)$.  Then there is some $y\in\widehat{\gamma}$ so that $|z-y|<r$.  
    Since $y\in\widehat{\gamma}$, $y=\gamma(t_y)$ for some $t_y$, and there is $t_{i_y}$ (which is the center of circle $C_{i_y}$) such that 
\[|t_y-t_{i_y}|\le \frac{1}{2\lceil 2^{k_\alpha(r)}10^{1/\alpha}C_\gamma^{1/\alpha}\rceil }\le  \frac{1}{ 2\cdot 2^{k_\alpha(r)}10^{1/\alpha}C_\gamma^{1/\alpha}}.\]  
It follows from the H\"older assumption on $\gamma$ and the definition of $k_\alpha(r)$ that $|\gamma(t_y)-\gamma(t_{i_y})|<\frac{r}{10}$.  That is, $\gamma(t_y)$ is enclosed in the circle $C_{i_y}$.  

Moreover, we have that $\sum_i\mathcal{H}^1(C_i) \lesssim_\alpha r C_\gamma^{1/\alpha}2^{k_\alpha(r)}$. Recalling  $r^{\frac{\alpha-1}{\alpha}}= 2^{k_\alpha(r)(1- \alpha)}$, we obtain the desired conclusion. 
 
 To extend the result to $\mathbb{R}^n$, we replace each circle $C_i$ with a family of $C(n)$ circles $\mathcal{C}_{i,n}$ in planes $P_i+\gamma(t_i)$ where $\{P_i\}_{i=1}^{C(n)}\subset G(n,2)$ are planes at sufficiently close, equally spaced angles off of each coordinate axis in the sense of spherical coordinates; see Figure \ref{fig:my_label}.  We also require that 
 $$ \bigcup_{C \in \mathcal{C}_{i,n}} C \cap \bigcup_{C \in \mathcal{C}_{i+1,n}} C \neq \varnothing, \text{ for every } i. $$
  We postpone the proof of these facts to Lemma \ref{lem: appendix}, located in the appendix.
\end{proof}
 \begin{figure}
     \centering     \includegraphics[scale=.5]{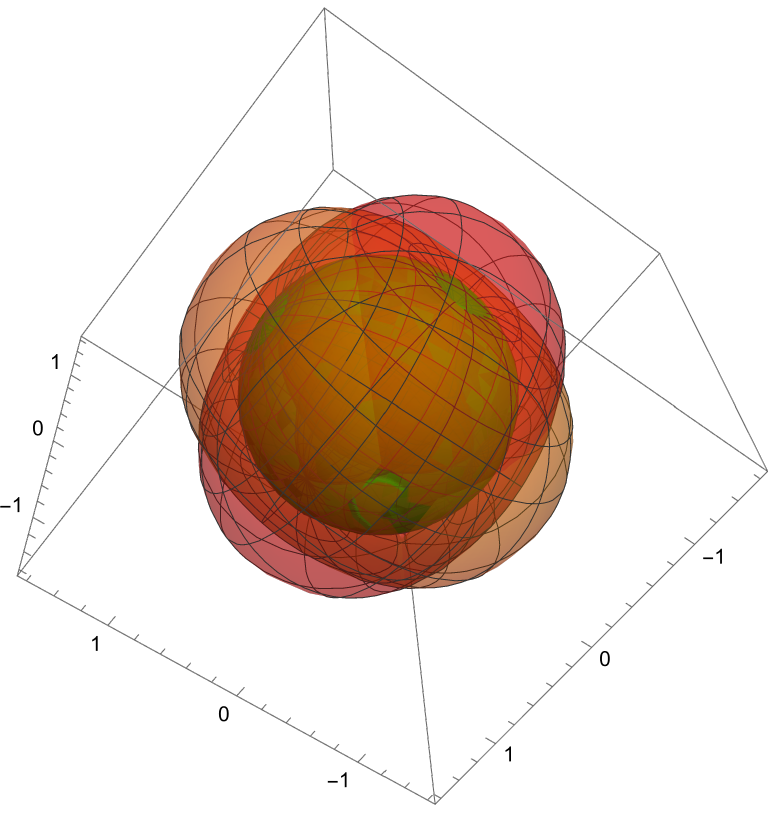}
     \caption{A sphere contained in the union of the $r$-neighborhoods of two circles.}
     \label{fig:my_label}
 \end{figure}
\begin{remark}\label{rem:general case B(E,r) to E}
    The arguments used throughout this section can be adapted to prove an analogous version of Theorem \ref{thm: general case} for $\widehat{\gamma}$ (instead of $B(\widehat{\gamma},r)$). 
    The upper bound follows simply from the fact that $\Lambda(A, r) \leq \Lambda(B, r)$ whenever $A \subset B$, and the only necessary substantial modification for the lower-bound would appear in Lemma \ref{lem:diam estimate}, where the lower bound would be $\frac{1}{2}\diam(E)$, provided that $r$ is small enough in comparison with $\diam(E)$.
\end{remark}

\section{Convergence of minimizers}\label{sec:Convergence of minimizers}

As stated in Section \ref{sec: intro}, one of the main goals of the paper is to answer the following question: given that $E$ is a curve, do the $r$-maximum distance minimizers converge to $E$ in Hausdorff distance as $r$ approaches $0$? 
In the contexts of $\Lambda(B(E,r),r)$ and $\Lambda(E, r)$, we answer this question affirmatively.

\begin{thm}\label{thm: Hausdorff conv H1}
    Let $E \subset \R^2$ be path connected. 
     For any $r>0$, we let $\Gamma^\ast_r$ be an $r$-maximum distance minimizer of $B(E, r)$. 
     Then  $$B(\Gamma^\ast_r,r)\subset B(E,Cr),$$
     where $C>1$ is an absolute constant.
\end{thm}

\begin{cor}\label{cor: Hausdorff conv length}
    Under the assumptions above, we have that $\Gamma^\ast_r$ converges to $E$ in Hausdorff distance as $r \rightarrow 0^+$.
\end{cor}

\begin{remark}
    We point out that in Theorem \ref{thm: Hausdorff conv H1}, the condition regarding path connectedness is natural, since it is clear that the conclusion fails for compact sets with several connected components.
\end{remark}

For the proof of the Theorem \ref{thm: Hausdorff conv H1}, we need the following lemma:
\begin{lem}\label{chain}
   Let $E$ and $r$ be as above and let $\Gamma^\ast_r$ be an $r$-maximum distance minimizer of $B(E, r)$. 
   Assume the existence of two elements $x_0$ and $x_1$ of ${\Gamma^\ast_r}\cap B(E,2r)$. 
   Then there exists a finite sequence of points $\{x_i\}_{i=2}^N \subset \Gamma^\ast_r\cap B(E,2r)$ satisfying that 
    $$|x_j-x_{j+1}|\leq 4r,$$
    for $j=1,...,N$, where we have identified $x_{N+1}\equiv x_{0}$.
\end{lem}

\begin{proof}
For $i\in \{0,1\}$, since  $\dist(x_i,E)\leq 2r$, there exists $y_i \in E$ such that 
$$|y_i-x_i|\leq 2r.$$
From here we can deduce, for some $n \in \mathbb{N}$, the existence of points $y_2,...,y_N \in E$ such that  
$$ 0<|y_j-y_{j+1}|\leq r\text{ for every }j=1,...,N \text{  with } y_{N+1}\equiv y_0.$$
Now, since $E \subset B(\Gamma^\ast_r,r)$, we can find $x_2,...,x_{N}\in \Gamma^\ast_r$ satisfying that 
$$|x_j-y_j|\leq r \text{ for }j=2,...,N.$$
Using the triangle inequality, we obtain the lemma.
\end{proof}

Additionally, for the proof of Theorem \ref{thm: Hausdorff conv H1}, it will be convenient to assume that our $r$-maximum distance minimizer is a tree. The following theorem will allow us to do so.

\begin{thm}[Theorem 5.5.~\cite{paolini2004qualitative}]\label{no loops theorem}
 Let $E\subset \R^n$ be a compact set, let $\Gamma^\ast_r$ be an $r$-maximum distance minimizer of $E$, and let $F_E(\Gamma^\ast_r) > 0$. 
 Then $\Gamma^\ast_r$ is a tree.
 \end{thm}

\subsection{Proof of Theorem \ref{thm: Hausdorff conv H1}}
 Let $E$ be a compact path connected set and let $r>0$. 
 Let $\Gamma^\ast_r$ be an $r$-maximum distance minimizer for $B(E, r)$ and let $ A> 1$ be a number to be determined later. 

For notational convenience and since $r$ remains fixed throughout the remainder of this section, we will denote $\Gamma^\ast_r$ simply by $\Gamma^\ast$.

Assume that $\Gamma^*\not\subset B(E,2Ar)$.  
As mentioned above, Theorem \ref{no loops theorem} allows us to deduce that $\Gamma^\ast$ is a tree. Indeed, we have that  $F_{B(E,r)}(\Gamma^\ast)>0$, since otherwise, $B(E,r) \subset \Gamma^\ast$ and therefore $\clH^1(\Gamma^\ast)=\infty$.
 
We begin by fixing a useful tree structure on $\Gamma^*$ and introducing relevant terminology. 
\begin{itemize}
    \item We fix a point $z\in {\Gamma^*}\setminus B(E,2Ar)$, to be the \textit{root} of $\Gamma^*$. In general, a root will be what we consider the starting point for ``pruning'' $\Gamma^\ast$ in the subsequent arguments. 

    \item Let $T\subset \Gamma^*$ be a sub-tree.  We let $x_T$ denote the unique point in $T$ such that for all $\epsilon>0$, $B(x_T,\epsilon)\cap (\Gamma^*\setminus T)\not=\varnothing$.  We call $x_T$ a \textit{sub-root} of the sub-tree $T$ and $T$ a \textit{loose branch} if $(\Gamma^*\setminus T)\cup \{x_T\}$ is connected.

    \item We say that a point $x\in \Gamma^*$ is a \emph{branch point} if $x=z$ is the root or if there exists $\epsilon_x>0$ such that for all $0<\epsilon<\epsilon_x$ we have that $\Gamma^* \setminus B(x,\epsilon)$ contains at least three components.

    \item For any $x,y\in {\Gamma^*}$, we define $\Gamma^*_{x,y}$ to be the unique directed path along ${\Gamma^*}$ from $x$ to $y$. For convenience, we do not include the endpoints $x$ and $y$ in $\Gamma^*_{x,y}$. 
    
    \item Let $x$ and $y$ be consecutive branch points. Then we call $\Gamma^*_{x,y}$ the \textit{limb} from $x$ to $y$.

    \item We say that a sub-tree $T \subset \Gamma^*$ is  \textit{non-covering} if $T \cap B(E,2r) = \varnothing$.

    \item Recall that $z$ denotes the root of $\Gamma^\ast$, and let $y \in \Gamma^\ast$. 
    Along the path $\Gamma^*_{z,y}$, we may enumerate the branch points as we move from $z$ to $y$; each of these branch points in $\Gamma^*_{z,y}$ is called a \textit{descendant} of $z$. 
    Moreover, we say $x$ is a \textit{child} of $z$ if $x$ is the first branch point encountered along $\Gamma^*_{z,y}$.  
    Similarly, for $w$ a descendant of $z$, we define $w'$ to be a child of $w$ if $w'$ is the first branch point encountered along $\Gamma^*_{z,y}$ after leaving $w$.

     \item For a branch point $x$ with child $x'$ in a sub-tree $T\subset \Gamma^\ast$, we say that we \emph{erase the descendant path} of $x$ associated to $x'$ to mean that we remove all of the following from $T$:  the limb $\Gamma_{x,x'}$, $x'$ itself, the limb for each descendant of $x'$, and all descendants of $x'$.
\end{itemize}

\begin{lem}[\textbf{Pruning the tree}]\label{lem: pruning}
The minimizer $\Gamma^*$ satisfies
    \begin{enumerate}
        \item[(P1)] $\Gamma^*$ has no non-covering loose branches,
        \item[(P2)]  every limb $\eta \subset \Gamma^\ast\setminus B(E,6r)$ satisfies that $\mathcal{H}^1(\eta)\leq4r$, and 
        \item[(P3)] any sub-tree $T \subset \Gamma^*$ has at most a finite number of branch points in $\Gamma^*\setminus B(E,6r)$.
    \end{enumerate}
\end{lem}

\begin{proof}
We prove the three properties sequentially.
    
    \textit{(P1) Discarding non-covering loose branches}: Assume that there is a non-covering loose branch $T\subset {\Gamma^\ast}$ with sub-root $x_T$ and $\mathcal{H}^1(T)>0$, otherwise there is nothing to prove. In this case, we may remove $T$ by considering $\Gamma'= ({\Gamma^\ast}\setminus T)\cup \{x_T\}$ in place of ${\Gamma^\ast}$. We note that this will maintain $E\subset B(\Gamma',2r)$ and $\mathcal{H}^1({\Gamma'})<\mathcal{H}^1({\Gamma^\ast})$. This clearly contradicts $\Gamma^*$ as a minimizer, and so we cannot have any non-covering loose branches.

    \textit{(P2) Discarding non-covering long limbs}: Assume the existence of a limb $\eta\subset {\Gamma^*}\setminus B(E,6r)$ such that $\mathcal{H}^1(\eta)>4r$. 
    Removing $\eta$ leaves at most two connected components of ${\Gamma^*}\setminus \eta$, say $\Gamma_0$ and $\Gamma_1$. 
    If we have that $\Gamma_i \subset B(E,2r)^c$ for  $i=0$ or 1, we simply remove $\Gamma_i$ and define $\Gamma' :=\Gamma_{i+1\,(\mathrm{mod}\,2)}$.  
    
    Otherwise, there exist points $x_0\in \Gamma_0$ and $x_1\in \Gamma_1$ so that $x_0,x_1\in B(E,2r)$. 
    We may use Lemma \ref{chain} to  find a sequence of points $x_2, ..., x_N \in {\Gamma^\ast}\cap B(E,2r)$ such that $|x_j - x_{j+1}|\leq 4r$ for $j=1,...,N$, where we have identified $x_0 \equiv x_{N+1}$.  
    We observe that each $x_j \in \Gamma_i$ for some $i=0,1$. 
    In particular, there exists some $j$ with $1 \leq j \leq N$ such that $x_j \in \Gamma_1$ and $x_{j+1} \in \Gamma_0$. 
    We proceed to connect $x_j$ and $x_{j+1}$ with the line segment $[x_j,x_{j+1}]$ whose length is at most $4r$. 
    We then define $\Gamma':= \Gamma_0\cup\Gamma_1\cup [x_j,x_{j+1}]$.
    
    Thus, in either case, we have obtained a connected curve $\Gamma'$ satisfying $$\mathcal{H}^1(\Gamma')<\mathcal{H}^1({\Gamma^\ast}).$$ 
    Once more, this contradicts the fact that $\Gamma^\ast$ is an $r$-maximum distance minimizer of $B(E, r)$.

     \textit{(P3) Discarding infinite branch points}: Assume that $T\subset {\Gamma^*}$ is a sub-tree with an infinite number of branch points that lie in $B(E,6r)^c$. We note that we have already discarded non-covering loose branches with possibly infinite branching; hence we may assume that there are infinitely many branch points $x_j\in T \setminus B(E,6r)$ such that there exists a path from $x_j$ to $\partial B(E,2r)$. In particular, each of these paths must have length at least $4r$, which contradicts the finite length of $\Gamma^*$.

     This completes the proof of Lemma \ref{lem: pruning}.
\end{proof}

With the assurance that $\Gamma^*$ satisfies (P1), (P2), and (P3), we turn to another lemma which allows us to estimate the $\mathcal{H}^1$ measure of $\Gamma^\ast$ from below. We will use the terminology developed above and the conclusions of Lemma \ref{lem: pruning}.

\begin{lem}[Identifying a binary sub-tree of sufficiently large $\clH^1$ measure]\label{lem: long substree}  Let $A\ge 64$. Assume that $\Gamma^*\not\subset B(E,2Ar)$ and define $T=\Gamma^*\cap B(z,Ar)$.  Then there exists a binary sub-tree $T'\subset T$ such that $\clH^1(T)\ge\clH^1(T')\ge 2^{A/8}4r$.
\end{lem}
\begin{proof}
   We begin by finding a binary sub-tree $T'\subset T$ where each limb $\eta\subset T'$ satisfies $4r< \clH^1(\eta)< 8r$. 

   Recall that $z$ is the root of $\Gamma^*$.   By (P1) together with (P2), $z$ is guaranteed to have at least two children.   
   We initialize $v=z$ to be a \textit{base point}.
   
    \textit{Base Point Step:} For base point $v$, we choose two children of $v$, which we label as $c_1$ and $c_2$.  We keep $T_{v,c_1}$ and $T_{v,c_2}$ but erase the descendant paths of all other children of $v$.  We then continue to the Child Step.

   \textit{Child Step:} For each $c_i$, $i\in\{1,2\}$, associated to $v$, we present two alternatives that dictate how to proceed.
    
    \indent \textit{Child Step Alternative 1:} When $\mathcal{H}^1(T_{v,c_i})\leq 4r$ the limb $T_{v,c_i}$ is not sufficiently long.  Choose one child of $c_i$, which we denote by $c_i'$, and erase the descendant paths of all other children of $c_i$. Then reset $c_i$ to be $c_i'$ and return to the \textit{Child Step}. 
    
    \indent  \textit{Child Step Alternative 2:} When $\mathcal{H}^1(T_{v,c_i})> 4r$ the limb $T_{v,{c_i}}$ is sufficiently long. Declare $c_i$ to be the base point $v$ and proceed to the \textit{Base Point Step}.

    We iterate this process until we do not have sufficiently many children available to complete the step.  
    We call the resulting sub-tree $T'$. The algorithm guarantees that each limb in $T'$ has $\clH^1$ measure between $4r$ and $8r$, as desired.

    Due to the above discussion (along with Lemma \ref{lem: pruning}), we guarantee that every unique path in $T'$ from $z$ will hit the boundary of the ball and each unique path will encounter at least $A/8$ branch points. This, in turn, implies that our binary tree has at least $2^{A/8}$ branch points. This fact, together with the length lower bound of each of the limbs, provides $\mathcal{H}^1(T')\geq 2^{A/8}4r.$
    \end{proof}

We are to ready to wrap up the proof of Theorem \ref{thm: Hausdorff conv H1}.
    
  We place a ball $B$ of radius $Ar$ centered at the root $z$. At the same time, we remove ${\Gamma^\ast} \cap B^\circ$, where $B^\circ$ denotes the interior of $B$. We notice that $ \Gamma':= ({\Gamma^*}\setminus B^\circ)\cup \partial B$ is still connected and that $B(E,r)\subset B(\Gamma',r)$. To complete the proof, it suffices to show that   $\mathcal{H}^1(\Gamma')< \mathcal{H}^1({\Gamma^\ast})$, thus providing a contradiction to $\Gamma^*$ being a minimizer.

    For that purpose, we apply Lemma \ref{lem: long substree}, obtaining  that $\clH^1({\Gamma^\ast} \cap B^\circ)\geq2^{A/8}4r$.  Thus we must simply choose $A$ big enough so that 
    $$2^{A/8}4r > 2\pi Ar.$$  
   Consequently, the theorem follows with $C=2A$.

\begin{remark}
    The authors believe that the arguments above that allowed us to prove Theorem \ref{thm: Hausdorff conv H1} would adapt without substantial modifications to higher dimension settings.
\end{remark}
\begin{remark}\label{rem:Convergence of minimizers B(E,r) to E}
    The arguments used throughout Section \ref{sec:Convergence of minimizers} can be adapted without difficulty to prove an analogous version of Theorem \ref{thm: Hausdorff conv H1} for $E$ (instead of $B(E,r)$) under the additional assumption that $E$ is not contained in a rectifiable curve.
\end{remark}
\begin{remark}
    It would be interesting to obtain an analogous result for $r$-maximum distance minimizers, replacing the $\mathcal{H}^1$ measure with the length of a curve $\ell$. The authors plan to address this problem in future publications.
\end{remark}

\appendix
\section{Details from Lemma \ref{lem: upper bound}}\label{sec:Variation}

\begin{lem} \label{lem: appendix}
There exist $C(n)$ circles of radius $\frac{r}{2}$, $\{C_i\}_{i=1}^{C(n)}$ in $\bbR^n$, such that if $x\in B\Big(0,\frac{11}{10}r\Big)$ then $x\in B(C_j,r)$ for some $j$.  
\end{lem}

\begin{proof}  
Fix $x\in B(0,\frac{11}{10}r)$, so 
\[x = (\rho\cos(\phi_1), \rho\sin(\phi_1)\cos(\phi_2),\dots, \rho\sin(\phi_1)\dots\cos(\phi_{n-1}), \rho\sin(\phi_1)\dots\sin(\phi_{n-1})),\]
with $0\leq \rho\leq \frac{11}{10}r$.
Suppose that the angles $\{\widetilde\phi_j\}_{j=2}^{n-1}$ are picked sufficiently close to $\phi_j$ so that $$0\le |\sin(\phi_j)-\sin(\widetilde{\phi}
_j)|<\epsilon \text{ and } 0\le |\cos(\phi_j)-\cos(\widetilde\phi_j)|<\epsilon,$$ for every $2 \leq j \leq n-1$.  Then choose
the plane
\[
P(t, \phi)=( t\cos(\phi), t\sin(\phi)\cos(\widetilde\phi_2),\dots, t\sin(\phi)\dots\cos(\widetilde\phi_{n-1}), t\sin(\phi)\dots\sin(\widetilde\phi_{n-1})), \]
with $0\le\phi\le 2\pi, t>0$, and the point $p\in P$ with coordinates
\[p=\Big(\frac{r}{2}\cos(\phi_1), \frac{r}{2}\sin(\phi_1)\cos(\widetilde\phi_2),\dots, \frac{r}{2}\sin(\phi_1)\dots\cos(\widetilde\phi_{n-1}), \frac{r}{2}\sin(\phi_1)\dots\sin(\widetilde\phi_{n-1})\Big).\]
We note that $p$ lives inside the circle of radius $r/2$ within $P$. Then 
$$|x-p|=\Big|\frac{r}{2}-\rho\Big| + r \cdot O(\varepsilon).$$
Finally, we notice that if we choose $\epsilon$ small enough depending on $n$  (i.e. we choose enough planes), we achieve the desired result. \end{proof}

\bibliographystyle{plain}
\bibliography{bibliography.bib}

\end{document}